\documentclass[12pt]{amsart}

\newtheorem{theorem}{Theorem}[section]
\newtheorem{lemma}[theorem]{Lemma}
\newtheorem{conjecture}[theorem]{Conjecture}

\theoremstyle{remark}
\newtheorem{remark}[theorem]{Remark}
\newcommand{\bP}{\mathbb{P}}
\newcommand{\E}{Erd\H{o}s}
\newcommand\cE{\mathbb{E}}
\DeclareMathOperator{\supp}{supp}

\usepackage{mathtools}

\DeclarePairedDelimiter\floor{\lfloor}{\rfloor}

\numberwithin{equation}{section}

\begin{document}

\title[Maximal tail of sums of nonnegative i.i.d. random variables]{On maximal tail probability of sums of nonnegative, independent and identically distributed random variables}

\author[T. \L{}uczak]{Tomasz \L{u}czak}
\address{Adam Mickiewicz University,
Faculty of Mathematics and Computer Science, 
ul.~Umultowska 87,
61-614 Pozna\'n, Poland}
\email{tomasz@amu.edu.pl}
\thanks{The first author is partially supported by NCN grant 2012/06/A/ST1/00261.}

\author[K. Mieczkowska]{Katarzyna Mieczkowska}
\address{Adam Mickiewicz University,
Faculty of Mathematics and Computer Science, 
ul.~Umultowska 87,
61-614 Pozna\'n, Poland}
\email{kaska@amu.edu.pl}
\thanks{The second author is partially supported by NCN grant 2012/05/N/ST1/02773.}

\author[M. \v Sileikis]{Matas \v{S}ileikis}
\address{University of Oxford, Mathematical Institute, Woodstock Road, Oxford OX2 6GG, United Kingdom}
\email{matas.sileikis@gmail.com}

\date{December 14, 2015}

\begin{abstract}
We consider the problem of finding the optimal upper bound for the tail probability of a sum of
$k$ nonnegative, independent and identically distributed random variables with given mean~$x$.
For $k=1$ the answer is given by Markov's inequality and
for $k=2$ the solution was found by Hoeffding and Shrikhande in 1955.
We solve the problem for $k=3$ as well as for general $k$ and $x\leq1/(2k-1)$ by showing that it follows from the fractional version of an extremal graph theory problem 
of Erd\H{o}s on matchings in hypergraphs.
\end{abstract}

\maketitle

\section{Introduction}
The purpose of this paper is to consider the problem of finding, for $x, t \ge 0$, the quantity 
\begin{equation}\label{eq:suptail}
	\sup_{\mathbf{X}}{\bP(X_1+\ldots+X_k\geq t)},
\end{equation}
where the supremum is taken over all random vectors $\mathbf{X}=(X_1,\ldots,X_k)$ 
of nonnegative, independent and identically distributed (further \emph{i.i.d.}) random variables 
$X_i$ such that $\mathbb{E}(X_i) \le x$ for $i=1,\ldots,k$.

From now on we assume that $t=1$, since, writing 
\begin{equation*}
	m_k(x) = \sup_{\mathbf{X}}{\bP(X_1+\ldots+X_k\geq 1)},
\end{equation*}
by rescaling we get that \eqref{eq:suptail} is equal to $m_k(x/t)$.

For $x \ge 1/k$ the trivial solution $m_k(x) = 1$ is given by $X_i$'s which are identically equal to $x$. 
For $k=1$ and $x < 1$ the solution $m_1(x) = x$ is given by Markov's inequality and a zero-one random variable.
In the case of two variables the problem was solved by Hoeffding and Shrikhande~\cite{HS} who showed that

$$m_2(x)=
\begin{cases} 
2x-x^2 & \text{for}\quad x < 2/5; \\
4x^2 & \text{for}\quad 2/5 \leq x <  1/2;\\
	1, & \text{for}\quad x \ge 1/2.
\end{cases}$$

We conjecture that the following generalization of the above results holds. 

\begin{conjecture}\label{conjLMS}
For every positive integer $k$ and $x\geq 0$ we have
\begin{equation}\label{eq:LMS}
	m_k(x)=
	\begin{cases} 
		1-(1-x)^k & \text{for}\quad x < x_0 ,\\
		(kx)^k & \text{for}\quad  x_0 \le x < 1/k,\\
		1, & \text{for}\quad  x\geq 1/k,
	\end{cases}
\end{equation}
where 
$x_0(k)$ is the solution of $1-(1-x)^k = (kx)^k$.
\end{conjecture}

Note that the lower bound on $m_k(x)$ in the first case is given by $X_i$'s with a two-point distribution 
$	\bP(X_i=1)=1-\bP(X_i=0)=x$
and in the second case by $X_i$'s with distribution
$	\bP(X_i=1/k)=1-\bP(X_i=0)=kx .$

The case when $X_i$'s are not necessarily identically distributed has also been studied. Let
\begin{equation*}
	s_k(x) = \sup_{\mathbf X} \bP\left( X_1 + \dots + X_k \ge 1\right),
\end{equation*}
where the supremum is taken over all vectors of independent, nonnegative random variables with common mean $x$. Clearly, $m_k(x) \le s_k(x)$. In 1966 Samuels \cite{S66} formulated a conjecture on the least upper bound for the tail probability in terms of $\cE X_i, i = 1, \dots, k$, which are not necessarily equal. For simplicity, we state this conjecture in the case when means are equal.
\begin{conjecture}[Samuels \cite{S66}]\label{conj:Samuels}
For every positive integer $k$ and $x\geq 0$ we have
	\begin{equation}\label{eq:Samuels}
		s_k(x)=
		\begin{cases} 
			1 - \min_{t = 0}^{k-1} \left(1 - \frac{x}{1 - tx} \right)^{k-t} & \text{for}\quad  x < 1/k,\\
			1, & \text{for}\quad  x\geq 1/k,
		\end{cases}
	\end{equation}
\end{conjecture}

The lower bound on $s_k(x)$ is given by one of the random vectors $\mathbf X_t, t=0, \cdots, k-1$, where $\mathbf X_t$ consists of~$t$ random variables identically equal to $x$ and $k-t$ i.i.d. random variables taking values $0$ and $1 - tx$. Therefore, if true, \eqref{eq:Samuels} implies Conjecture \ref{conjLMS} only when the minimum is attained by $t=0$. Computer-generated graphs suggest that the minimum is attained by $t=0$ when $x < x_1(k)$ and by $t= k - 1$ when $x \ge x_1(k)$, where $x_1(k) \in (0, 1/k)$ is the solution of $1-(1 - x)^k = x/ (1 - (k-1)x)$. In \cite{AFHRRS} it was shown rigorously that the minimum is attained by $t=0$ for $x \le 1/(k+1)$.

Samuels \cite{S66, S68} confirmed \eqref{eq:Samuels} for $k = 3, 4$. Computer-generated graphs of functions $s_3(x)$ and $s_4(x)$ suggest that for $k = 3,4$ we have
	\label{thm:Samuels}
	\begin{equation*}
		m_k(x)= s_k(x) = 1 - (1 -x)^k, \qquad x \le x_1(k),	
	\end{equation*}
	where $x_1(3) = 0.27729\dots$, $x_1(4) = 0.21737\dots$.

	Moreover, Samuels \cite{S69} proved that for $k \ge 5$
	\begin{equation*}
		m_k(x) = s_k(x) = 1 - (1-x)^k, \qquad x \le 1/(k^2-k).
	\end{equation*}

Our main result can be stated as follows.

\begin{theorem}\label{thmLMS}
	Conjecture~\ref{conjLMS} holds for $k=3$ and every $x$. Moreover, it holds for $k \ge 5$ when $x < \frac{1}{2k-1}$.
\end{theorem}

The proof of Theorem~\ref{thmLMS} is based on an observation that 
Conjecture~\ref{conjLMS} is asymptotically equivalent 
to the fractional version of Erd\H{o}s' Conjecture 
on matchings in hypergraphs which we introduce in the next section.

\section{The hypergraph matching problem}\label{sec2}

A {\em $k$-uniform hypergraph}  $H=(V,E)$  
is a~set of vertices $V$ 
together with a~family $E$ of 
$k$-element subsets of $V$, called \emph{edges}.
A {\em matching} is a~family of disjoint edges of $H$, and the size 
of the largest matching in $H$ is called a~{\em matching number} and is denoted by $\nu(H)$. 

In~\cite{E} \E\ stated the following.

\begin{conjecture}[Erd\H{o}s~\cite{E}]\label{conjerd}
Let $H=(V,E)$ be a~$k$-uniform hypergraph, $|V|=n$, $\nu(H) = s$. 
If $n\geq ks+k-1$, then 
\begin{equation}\label{eq:erdos_conj}
|E|\leq \max\bigg\{\binom nk-\binom{n-s}k\,,\binom {ks+k-1}k\bigg\}.
\end{equation}
\end{conjecture}

Note that the equality in Conjecture~\ref{conjerd} holds either when $H$ is 
a~hypergraph consisting of
all $k$-element sets intersecting a~given subset $S\subset V$, $|S|=s$,
or when $H$ consists of all $k$-element subsets of a given subset $T\subset V$, $|T|=ks+k-1$.
We denote these two families of hypergraphs by $Cov_{n,k}(s)$ and $Cl_{n,k}(ks+k-1)$, respectively.

A similar problem can be formulated in terms of fractional matchings.
A~\emph{fractional matching} in a~hypergraph $H$ 
is a~function 
$$
\begin{aligned}
&w:E\rightarrow [0,1] \text{ such that }\\ 
\sum_{e\ni v}{w(e)}&\leq 1  \text{ for every vertex } v\in V.
\end{aligned}$$
Then, $\sum_{e\in E}{w(e)}$ is a~{\em size} of the matching $w$ and 
a~size of the largest fractional matching in $H$, denoted by $\nu^*(H)$,
is a~{\em fractional matching number}. 
Alon, Frankl, Huang, R\"odl, Ruci\'nski and Sudakov~\cite{AFHRRS} stated the following conjecture.

\begin{conjecture}[\cite{AFHRRS}]\label{conjfrac}
	Let $x \in [0,1/k]$ be fixed and let $H_n=(V_n,E_n)$ be a sequence of $k$-uniform hypergraphs such that $ \nu^*(H_n) \le x|V_n|$.
Then
\begin{equation}\label{eq:limsup}
	\limsup_{n\to \infty} \frac{ |E_n|}{{|V_n| \choose k}} \le \max\left\{1-(1-x)^k, {(kx)^k}\right\}.
\end{equation}
\end{conjecture}

Finding a~fractional matching number 
is a~linear programming problem. 
Its dual problem is to minimize the size of a \emph{fractional vertex cover} of $H$, which is defined as a~function 
\[
	\begin{aligned}
		&w: V \rightarrow[0, 1] \text{ such that }\\ 
		\text{ for each } &e\in E \text{ we have } \sum_{v\in e}{w(v)}\geq 1.
	\end{aligned}
\]
Then, $\sum_{v\in V}{w(v)}$ is the~\emph{size} of $w$ and 
the size of the smallest fractional vertex cover in $H$ is denoted by $\tau^*(H)$. 
By the Duality Theorem, 
\begin{equation}\label{eq:nutau}
	\nu^*(H)=\tau^*(H).
\end{equation}

The bound in Conjecture~\ref{conjfrac}, if true, is 
attained by either a sequence $H_n \in Cov_{n,k}(\floor{xn})$ (which has a fractional vertex cover $w(v)= \mathbf 1_{v \in S}$ of size $\floor{xn}$ and therefore satisfies $\nu^*(H_n) \le xn$) or $H_n \in Cl_{n,k}(\floor{kxn})$ (which has fractional vertex cover $w(v) = \frac 1 k \mathbf 1_{v \in T}$ of size $\floor{kxn}/k$).

Observe that if a fractional matching $w$ is such that $w(e)\in \{0,1\}$ for every edge~$e$, 
then $w$ is just a~matching or, more precisely, the indicator function of a~matching.
Thus, every integral matching is also a~fractional matching 
and hence
\begin{equation}\label{eq:nunu}
	\nu(H)\leq \nu^*(H),
\end{equation}
so consequently, Conjecture~\ref{conjfrac} follows from Conjecture~\ref{conjerd}. Furthermore,
Conjecture~\ref{conjerd} was confirmed for $k = 3$ by the first two authors \cite{LM} (for $n$ bigger than some absolute constant) and by Frankl \cite{Fnew} (for every $n$). 
Moreover, Frankl \cite{Fgen} confirmed Conjecture \ref{conjerd} for $k \ge 4$ and $s \le (n-k)/(2k-1)$. Therefore, in view of \eqref{eq:nunu} we have the following.
\begin{remark}[\cite{Fnew,Fgen,LM}]\label{thm:frac}
	Conjecture~\ref{conjfrac} holds for $k = 3$ and every $x$ as well as for $k \ge 4$ and $x < 1/(2k-1)$.
\end{remark}

\section{Proof of Theorem~\ref{thmLMS}}

We  prove Theorem~\ref{thmLMS} in two steps. 
First we observe that it is enough to confirm
Conjecture~\ref{conjLMS} with some additional restrictions.
Then we show the equivalence of Conjectures~\ref{conjLMS} and~\ref{conjfrac}.

Here and below, given a discrete random variable, $\supp(X)$ denotes the set of values which $X$ attains with positive probability.

\begin{lemma}\label{lm1}
	It suffices to prove Conjecture~\ref{conjLMS} for $X_i$'s with discrete distribution satisfying the following properties: (i) $\supp(X_i)$ is finite subset of $[0,1]$; (ii) $\bP\left( X_i = a \right) \in \mathbb Q$ for every $a \in \supp(X_i)$.
\end{lemma}

\begin{proof}

Define, for $x \ge 0$, 
$$M(x)=\max\{1-(1-x)^k, (kx)^k, 1\},$$
which is equal to the right-hand side of \eqref{eq:LMS}.

Let us first assume that Conjecture~\ref{conjLMS} holds for random variables
satisfying (i) and (ii) and show it holds for $X_i$'s satisfying (i) only, that is, with $\supp(X_i)=\{a_1,\ldots,a_m\} \subset [0,1]$.
Let $p_j=\bP\left(X_i=a_j\right)$, $j=1,\ldots,m$.
For every sufficiently large integer $n$ define a random variable $Y_i^{(n)}$ such that
$$\bP\left(Y_i^{(n)}=a_j\right)=p_j^{(n)} \in \mathbb Q, \quad j = 1, \dots, m$$
where $p_j^{(n)}=\lceil n p_j\rceil /n$ for $j=2,\ldots,m$,
and $p_1^{(n)}=1-\sum_{j=2}^m p_j^{(n)}$ (note that the $p_1^{(n)}$ is positive for $n$ large enough).
Then, for every $j$ we have
$p_j^{(n)}\leq p_j+1/n,$ and therefore
$$\cE(Y_i^{(n)})\leq \cE(X_i)+ m/n.$$
Applying the conclusion of Conjecture 1 to $Y^{(n)}_i$'s  and using the continuity of function $M$, we get
\begin{align*}
	\bP\left(\textstyle \sum_{i=1}^kX_i\geq 1\right) &= \lim_{n \to \infty} \bP\left(\textstyle\sum_{i=1}^kY_i^{(n)}\geq 1\right) 
&	\le \lim_{n \to \infty} M(x + m/n) = M(x).
\end{align*}
Let us now assume that the Conjecture~\ref{conjLMS} holds for random variables satisfying (i) and show it then holds for arbitrary $X_i, i = 1, \dots, k$. For every positive integer $m$ define random variables
$$Y_i^{(m)}=\min \{\left\lceil m X_i\right\rceil/{m}, 1\}.$$
Note that $\supp(Y_i^{(m)})$ is finite and contained in $[0,1]$.
We have $Y_i^{(m)}\leq X_i + 1/m$ and therefore $\cE(Y_i^{(m)})\leq \cE(X_i) + 1/m \le x + 1/m$. For sufficiently large $m$ we get 
\begin{align*}
	\bP\left(\textstyle \sum_{i=1}^kX_i\geq 1\right) &\leq \bP \left( \textstyle \sum_{i=1}^k \lceil m X_i \rceil/m \ge 1 \right) = \bP\left(\textstyle \sum_{i=1}^kY_i^{(m)}\geq 1\right) \leq M(x+1/m).
\end{align*}
Taking a limit over $m \to \infty$ and using the fact that $M$ is continuous (from the right), we obtain 
$$\bP\left(\textstyle \sum_{i=1}^kX_i\geq 1\right)\leq M(x).\quad \qed$$
\renewcommand{\qed}{}
\end{proof}

\begin{lemma}\label{lm2}
	For every $k$ and $x \in [0,1/k]$ Conjectures~\ref{conjLMS} and~\ref{conjfrac} are equivalent.
\end{lemma}

\begin{proof}
	The proof that Conjecture~\ref{conjLMS} implies Conjecture~\ref{conjfrac} goes along the same lines as the proof of Theorem 2.1 in~\cite{AFHRRS}.
We recall it below for the sake of completeness.

Let us fix $k$ and $x \in [0,1/k]$ and suppose that Conjecture~\ref{conjLMS} holds.
Moreover, let $H_n=(V_n,E_n)$ be a~sequence of $k$-uniform hypergraphs such that $\nu^*(H_n)\leq x|V_n|=xn$. 
By~(\ref{eq:nutau}) we have $\tau^*(H_n)=\nu^*(H_n)\leq xn$, 
hence there exists a weight function $w_n:V_n\rightarrow [0,1]$ such that 
$$\sum_{v\in V_n}{w_n(v)}=xn,$$
and $\sum_{v\in e}{w_n(v)}\geq 1$ for every $e\in E_n$.

Let $(v^{n}_1,\ldots,v^{n}_k) \in V_n^k$ be a~vector of random vertices,
each chosen independently and uniformly over $V_n$.
Note that $w_n(v^{n}_1),\ldots,w_n(v^{n}_k)$ are nonnegative, independent and identically distributed random variables with mean
\begin{equation*}
	\mathbb{E}(w_n(v^{n}_i))=\frac{1}{|V_n|}\sum_{v\in V_n}{w_n(v)}=\frac{1}{n}xn=x.
\end{equation*}
Observe also  that 
\begin{equation}\label{eq1}
\mathbb{P}(\{v^{n}_1,\ldots,v^{n}_k\}\in E_n)=\frac{k!|E_n|}{n^k}.\end{equation}
On the other hand, since $w_n$ is a vertex cover of $H_n$, for $\{v^{n}_1,\ldots,v^{n}_k\} \in E_n$ we have $\sum_{i=1}^k{w_n(v^{n}_i)}\geq 1$ and thus
\begin{equation}\label{eq2}
\mathbb{P}(\{v^{n}_1,\ldots,v^{n}_k\}\in E_n)\leq\bP\left(\textstyle \sum_{i=1}^k{w_n(v^{n}_i)}\geq 1\right).\end{equation}
From (\ref{eq1}), (\ref{eq2}) and 
the assumption that Conjecture~\ref{conjLMS} is true, we conclude that
\begin{equation*}
	\limsup_{n\to \infty} \frac{ |E_n|}{{|V_n| \choose k}} \le \limsup_{n\to \infty}  \bP\left(\textstyle \sum_{i=1}^k{w_n(v^{n}_i)}\geq 1\right)\leq 	\max\left\{1-(1-x)^k, {(kx)^k}\right\}.
\end{equation*}

It remains to prove the reverse implication.
Let us assume that Conjecture~\ref{conjfrac} is valid for some $k$ and $x \in [0,1/k]$.
Due to Lemma~\ref{lm1} it is enough to show that Conjecture~\ref{conjLMS} holds for $X_i$'s attaining a finite set of values $a_1, \dots, a_m \in [0,1]$ such that
$$\mathbb{P}(X_i=a_j)={p_j}/{q_j}, \qquad j = 1, \dots, m$$
for some positive integers $p_j$ and $q_j$. 
Moreover, let $r$ be the smallest common multiple of the numbers $\{q_1,\ldots,q_m\}$, and define integers 
\begin{equation}\label{eq:pprime}
	p'_j = rp_j/q_j, \qquad j = 1, \dots, m. 
\end{equation}

In order to apply Conjecture \ref{conjfrac}, we  define hypergraphs 
with bounded fractional matching number.
For $n=1,2,\ldots$, let $V_n=[nr]$. Observing that $np_1' + \dots + np_m' = nr$, define a function $w_n:V_n\rightarrow[0,1]$ in such a way that for each $j = 1, \dots, m$ function $w_n(v)$ takes value $a_j$ precisely $np_j'$ times.
Let $H_n=(V_n,E_n)$ be a hypergraph with the edge set 
$$E_n=\left\{e\in\binom{V_n}{k}: \sum_{v\in e}{w_n(v)}\geq 1\right\}.$$
In view of  \eqref{eq:pprime}, we have that $w_n$ is a~fractional
vertex cover of $H_n$ of size 
\begin{equation*}
	\sum_{v=1}^{nr}{w_n(v)}=\sum_{j=1}^m{a_j}n{p'_j} = n\sum_{j=1}^m ra_j\frac{p_j}{q_j} = nr \mathbb E (X_i) \le xnr.
\end{equation*}
Hence by \eqref{eq:nutau} we have $\nu^*(H_n)=\tau^*(H_n)\leq xnr$ and therefore  \eqref{eq:limsup} gives
\begin{equation}\label{eq:Ebound}
	\limsup_{n \to \infty}\frac{ |E_n|} {\binom {nr}k}\le \max \left\{ 1-(1-x)^k, (kx)^k \right\}.
\end{equation}

Let $(v^{n}_1,\ldots,v^{n}_k) \in V_n^k$ be a~vector of random vertices,
each chosen independently and uniformly over $V_n$.
Note that for every $n$ the random variable $w_n(v_i^{n})$ has the same distribution as $X_i$, since, by \eqref{eq:pprime}, 
$$\bP(w_n(v_i^{n}) = a_j)=\frac{np_j'}{|V_n|}=\frac{n r p_j/q_j}{nr}=\frac{p_j}{q_j}, \qquad j = 1, \dots, m.$$
Let $N_n$ denote the number of $k$-element vectors $(v_1,\ldots,v_k) \in V_n^k$
of vertices with at least two equal coordinates. We have 
\begin{align*}
\bP(X_1+\ldots +X_k\geq 1)=& \bP(w_n(v^{n}_1) + \dots + w_n(v^{n}_k) \geq 1)\\
=&\frac{\left|\left\{(v_1,\ldots,v_k)\in V_n^{k}: \sum_i{w(v_i)}\geq 1\right\}\right|}{(nr)^k}\\
=& \frac{k!|E_n|+N_n}{(nr)^k} \le \frac{|E_n|}{\binom {nr}k} + \frac{\binom {k} 2 (nr)^{k-1}}{(nr)^k}.
\end{align*}

Taking the limit over $n \to \infty$ and using \eqref{eq:Ebound} we get that
$$\mathbb{P}(X_1+\ldots +X_k\geq 1) \leq \max\{1-(1-x)^k,(kx)^k\}.\quad\qed$$
\renewcommand{\qed}{}
\end{proof}

Now Theorem~\ref{thmLMS} follows from Lemma~\ref{lm2}
and Remark \ref{thm:frac}.

\bibliographystyle{amsplain}

\end{document}